\documentclass[11pt]{amsart}
\usepackage{lmodern}
\usepackage{amsmath, amsthm, amssymb, amsfonts}
\usepackage[normalem]{ulem}
\usepackage{hyperref}

\usepackage{verbatim} 
\usepackage{longtable}

\usepackage{mathtools}

\usepackage{tikz}
\usetikzlibrary{decorations.pathmorphing}
\tikzset{snake it/.style={decorate, decoration=snake}}

\usepackage{caption}

\usepackage{tikz-cd}
\usetikzlibrary{arrows}

\theoremstyle{plain}
\newtheorem{thm}{Theorem}[section]
\newtheorem{cor}[thm]{Corollary}
\newtheorem{lem}[thm]{Lemma}
\newtheorem{prop}[thm]{Proposition}
\newtheorem{conj}[thm]{Conjecture}

\theoremstyle{definition}

\newtheorem{example}[thm]{Example}

\theoremstyle{remark}
\newtheorem{rmk}[thm]{Remark}

\newcommand{\BA}{{\mathbb{A}}}

\newcommand{\BC}{{\mathbb{C}}}

\newcommand{\BG}{{\mathbb{G}}}
\newcommand{\BH}{{\mathbb{H}}}

\newcommand{\BL}{{\mathbb{L}}}

\newcommand{\BN}{{\mathbb{N}}}

\newcommand{\BP}{{\mathbb{P}}}
\newcommand{\BQ}{{\mathbb{Q}}}

\newcommand{\BX}{{\mathbb{X}}}

\newcommand{\BZ}{{\mathbb{Z}}}

\newcommand{\CA}{{\mathcal A}}
\newcommand{\CB}{{\mathcal B}}
\newcommand{\CC}{{\mathcal C}}

\newcommand{\CE}{{\mathcal E}}
\newcommand{\CF}{{\mathcal F}}

\newcommand{\CH}{{\mathcal H}}
\newcommand{\CI}{{\mathcal I}}

\newcommand{\CK}{{\mathcal K}}
\newcommand{\CL}{{\mathcal L}}
\newcommand{\CM}{{\mathcal M}}
\newcommand{\CN}{{\mathcal N}}
\newcommand{\CO}{{\mathcal O}}
\newcommand{\CP}{{\mathcal P}}
\newcommand{\CQ}{{\mathcal Q}}

\newcommand{\CS}{{\mathcal S}}
\newcommand{\CT}{{\mathcal T}}
\newcommand{\CU}{{\mathcal U}}

\newcommand{\FM}{{\mathfrak{M}}}

\DeclareFontFamily{OT1}{rsfs}{}
\DeclareFontShape{OT1}{rsfs}{n}{it}{<-> rsfs10}{}
\DeclareMathAlphabet{\curly}{OT1}{rsfs}{n}{it}

\newcommand{\git}{\mathbin{
  \mathchoice{/\mkern-6mu/}
    {/\mkern-6mu/}
    {/\mkern-5mu/}
    {/\mkern-5mu/}}}

\newcommand{\Coh}{\mathrm{Coh}}

\usepackage{tikz}
\usepackage{lmodern}
\usetikzlibrary{decorations.pathmorphing}

\usepackage{amsmath}
\usepackage{mathabx}

\begin{document}
\title[The P=W conjecture for $\mathrm{GL}_n$]{The P=W conjecture for $\mathrm{GL}_n$}
\date{\today}

\author[D. Maulik]{Davesh Maulik}
\address{Massachusetts Institute of Technology}
\email{maulik@mit.edu}

\author[J. Shen]{Junliang Shen}
\address{Yale University}
\email{junliang.shen@yale.edu}

\begin{abstract}
We prove the $P=W$ conjecture for $\mathrm{GL}_n$ for all ranks $n$ and curves of arbitrary genus $g\geq 2$.  The proof combines a strong perversity result on tautological classes with the curious Hard Lefschetz theorem of Mellit. For the perversity statement, we apply the vanishing cycles constructions in our earlier work to global Springer theory in the sense of Yun, and prove a parabolic support theorem.

\end{abstract}

\baselineskip=14.5pt
\maketitle

\setcounter{tocdepth}{1} 

\tableofcontents
\setcounter{section}{-1}

\section{Introduction}

Throughout, we work over the complex numbers $\BC$. 

\subsection{The P=W conjecture}
The purpose of this paper is to present a proof of the $P=W$ conjecture by de Cataldo--Hausel--Migliorini \cite{dCHM1} for arbitrary rank $n$ and genus $g \geq 2$.

Let $C$ be a nonsingular irreducible projective curve of genus $g \geq 2$. For two coprime integers $n \in \BZ_{\geq 1}$ and $d \in \BZ$, there are two moduli spaces $M_{\mathrm{Dol}}$ and $M_B$, called the Dolbeault and the Betti moduli spaces, attached to $C,n,$ and $d$.

The Dolbeault moduli space parameterizes stable Higgs bundles $(\CE, \theta)$ on $C$, where $\CE$ is a vector bundle on $C$ of rank $n$ and degree $d$, $\theta: \CE \to \CE \otimes \Omega_C^1$ is a Higgs field, and stability is defined with respect to the slope $\mu(\CE,\theta) ={\mathrm{deg}(\CE)}/{\mathrm{rk(\CE)}}$. The moduli space $M_{\mathrm{Dol}}$ admits the structure of a completely integrable system
\[
h: M_{\mathrm{Dol}} \to A := \bigoplus_{i=1}^n H^0(C, {\Omega_C^{1}}^{\otimes i}), \quad (\CE, \theta) \mapsto \mathrm{char.polynomial}(\theta),
\]
which is referred to as the Hitchin system \cite{Hit, Hit1}. The Hitchin map $h$ is surjective and proper; it is also Lagrangian with respect to the canonical holomorphic symplectic form on $M_{\mathrm{Dol}}$ induced by the hyper-K\"ahler metric. The \emph{perverse filtration} is an increasing filtration 
\[
P_0H^*(M_{\mathrm{Dol}}, \BQ)\subset P_1H^*(M_{\mathrm{Dol}}, \BQ) \subset \cdots \subset H^*(M_{\mathrm{Dol}}, \BQ)
 \]
on the (singular) cohomology of $M_{\mathrm{Dol}}$ governed by the topology of the Hitchin system $h$; see Section \ref{sec1.1} for a brief review.

The Betti moduli space $M_B$ is the (twisted) character variety associated with $\mathrm{GL}_n(\BC)$ and degree $d$.  It parametrizes isomorphism classes of irreducible local systems
\[
\rho: \pi_1(C\backslash \{p\}) \rightarrow \mathrm{GL}_n(\BC)
\]
where $\rho$ sends a loop around a chosen point $p$ to 
$e^{\frac{2\pi \sqrt{-1} d}{n}}\mathrm{Id}_n$.  Concretely, we have
\begin{equation*}
M_B := \Big{\{}a_k, b_k \in \mathrm{GL}_n(\BC),~k=1,2,\dots,g: ~~\prod_{j=1}^g [a_j, b_j] = e^{\frac{2\pi \sqrt{-1} d}{n}}\mathrm{Id}_n \Big{\}}\git\mathrm{GL}_n(\BC).
\end{equation*}
It is an affine variety whose mixed Hodge structure admits a nontrivial weight filtration 
\[
W_0 H^*(M_B, \BQ) \subset W_1H^*(M_B, \BQ) \subset \cdots \subset H^*(M_B, \BQ).
\]

Non-abelian Hodge theory \cite{Simp, Si1994II} gives a diffeomorphism between the two very different algebraic varieties $M_{\mathrm{Dol}}$ and $M_B$, which canonically identifies their cohomology:
\begin{equation}\label{NAH}
 H^*(M_{\mathrm{Dol}}, \BQ)  = H^*(M_{B}, \BQ).
\end{equation}
The $P=W$ conjecture by de Cataldo--Hausel--Migliorini \cite{dCHM1} 
refines the identification (\ref{NAH}); it predicts that the perverse filtration associated with the Hitchin system is matched with the (double-indexed) weight filtration associated with the Betti moduli space. This establishes a surprising connection between topology of Hitchin systems and Hodge theory of character varieties.

\begin{conj}\label{conj}[The P=W conjecture for $\mathrm{GL}_n$, \cite{dCHM1}] For any $k,m\in \BZ_{\geq 0}$, we have
\[
P_kH^m(M_{\mathrm{Dol}}, \BQ) = W_{2k}H^m(M_{B}, \BQ) = W_{2k+1}H^m(M_B, \BQ).
\]
\end{conj}

Conjecture \ref{conj} has previously been proven for $n=2$ and arbitrary genus $g\geq 2$ by de Cataldo--Hausel--Migliorini \cite{dCHM1}, and recently for arbitrary rank $n$ and genus 2 by de Cataldo--Maulik--Shen \cite{dCMS}. The compatibility between the $P=W$ conjecture and Galois conjugation on the Betti side was proven in \cite{dCMSZ}, which implies that $P=W$ does not depend on the degree $d$ as long as it is coprime to $n$. We refer to \cite[Section 1]{dCHM1} and the paragraphs following \cite[Theorem 0.2]{dCMS} for discussions concerning connections between Conjecture \ref{conj} and other directions. In particular, by \cite{CDP} and \cite[Section 9.3]{MT}, Conjecture \ref{conj} implies the correspondence between Gopakumar--Vafa invariants and Pandharipande--Thomas invariants \cite[Conjecture 3.13]{MT} for the local Calabi--Yau 3-fold $T^*C \times \BC$ in any curve class $n[C]$.

The main result of this paper is a full proof of Conjecture \ref{conj}.

\begin{thm}\label{thm}
    Conjecture \ref{conj} holds.
\end{thm}

Conjecture \ref{conj} has several variants which, to the best of our knowledge, are still open. These include the version formulated for possibly singular moduli spaces, intersection cohomology, and general reductive groups \cite{dCHM1, dCM, FM}, and the version formulated for moduli stacks \cite{Davison}.  There also exist parabolic versions of Conjecture \ref{conj} and we expect that our argument applies to these settings using parabolic variants of the ingredients here \cite{Mellit, OY}. We refer to \cite{DMS, HLSY, Geo_P=W, Zhang} and references therein for the $P=W$ phenomenon in other settings.

For the case of type $A_{n-1}$ ($\mathrm{GL}_n, \mathrm{PGL}_n, \mathrm{SL}_n$) and a degree $d$ coprime to $n$, Conjecture \ref{conj} (\emph{i.e.} the $P=W$ conjecture for $\mathrm{GL_n}$) is equivalent to the $P=W$ conjecture for $\mathrm{PGL}_n$; see the paragraph after \cite[Theorem 0.2]{dCMS}. The case of $\mathrm{SL}_n$ is more subtle --- it is closely related to the endoscopic decomposition of the $\mathrm{SL}_n$ Hitchin moduli spaces. A systematic discussion on this aspect can be found in \cite[Section 5]{MS}. In the special case when $n=p$ a prime number, \cite[Theorem 0.2]{SLn} implies that the $P=W$ conjecture for $\mathrm{GL}_p$, $\mathrm{SL}_p$, and $\mathrm{PGL}_p$ are all equivalent. In particular, we have the following immediate consequence of Theorem \ref{thm}, which generalizes the result of \cite{dCHM1} for $\mathrm{SL}_2$.

\begin{thm}
The $P=W$ conjecture holds for a curve $C$ of any genus at least two and $\mathrm{SL}_p$ with $p$ a prime number.
\end{thm}

\subsection{Idea of the proof}
Our proof of Conjecture \ref{conj} has 4 major steps.

\begin{enumerate}
    \item[Step 1.] \emph{Strong perversity of Chern classes}: Using work of Markman \cite{Markman}, Shende \cite{Shende}, and Mellit \cite{Mellit}, we may reduce the $P=W$ conjecture to a statement about the interaction between Chern classes of the universal family and the perverse filtration. This allows us to reduce the $P=W$ conjecture to a sheaf-theoretic statement which concerns strong perversity of Chern classes.
    \item[Step 2.] \emph{Vanishing cycle techniques}: We apply the formalism of vanishing cycles to reduce the sheaf-theoretic formulation of Step 1 for the Hitchin system, to the case of twisted Hitchin systems associated with meromorphic Higgs bundles. Our motivation is from the key observation by Ng\^o \cite{Ngo} and Chaudouard-Laumon \cite{CL}, that the decomposition theorem for such twisted Hitchin systems is more manageable. 
    
    Steps 1 and 2 are carried out in Sections 1 and 2.  

    \item[Step 3.]\emph{Global Springer theory}: The global Springer theory of Yun \cite{Yun1, Yun2,Yun3} produces rich symmetries for certain Hitchin moduli spaces; this proves the sheaf-theoretic formulation of Step 1 over the \emph{elliptic locus}. For our purpose, we need a version of global Springer theory for the stable locus over the total Hitchin base. This is described in Section 3.

    \item[Step 4.] \emph{A support theorem}: Lastly, in order to extend part of Yun's symmetries induced by Chern classes from the elliptic locus to the total Hitchin base, we prove a support theorem parallel to \cite{CL} for certain parabolic Hitchin moduli spaces. This is completed in Section 4.
\end{enumerate}

The idea of combining vanishing cycle functors and support theorems was also applied in \cite{MS} to give a proof of the topological mirror symmetry conjecture of Hausel--Thaddeus \cite{HT, GWZ}.

\subsection{Acknowledgements}
We would like to thank Mark Andrea de Cataldo, Bhargav Bhatt, Ben Davison, Jochen Heinloth, and Max Lieblich for various discussions.  We are especially grateful to Zhiwei Yun for explaining his thesis to us, both in his office and at the playground.  We also thank the anonymous referee for careful
reading and useful suggestions on the exposition. J.S. was supported by the NSF grants DMS-2134315 and DMS-2301474.

\section{Perverse filtrations and vanishing cycles}

In this section, we introduce the notion of \emph{strong perversity} and show its compatibility with vanishing cycle functors (Proposition \ref{prop1.4}). This plays a crucial role in Section 2 in lifting the $P=W$ conjecture sheaf-theoretically and reduce it to the twisted case.

\subsection{Perverse filtrations}\label{sec1.1}
Let $f: X \to Y$ be a proper morphism between irreducible nonsingular quasi-projective varieties (or Deligne--Mumford stacks) with $\mathrm{dim}X = a$ and $\mathrm{dim}Y = b$. Let $r$ be the defect of semismallness of $f$:
\[
r: = \mathrm{dim} X \times_YX - \mathrm{dim}X.
\]
In particular, we have $r = a-b$ when $f$ has equi-dimensional fibers. The perverse filtration 
\[
P_0H^m(X, \BQ) \subset P_1H^m(X, \BQ) \subset \dots \subset H^m(X, \BQ)
\]
is an increasing filtration on the cohomology of $X$ governed by the topology of the morphism $f$; it is defined to be
\[
P_iH^m(X, \BQ) := \mathrm{Im}\left\{ H^{m-a+r}(Y, {^\mathbf{p}\tau_{\leq i}} (Rf_* \BQ_X[a-r])) \to H^m(X, \BQ)\right\}
\]
where $^\mathbf{p}\tau_{\leq * }$ is the perverse truncation functor \cite{BBD}. 

\begin{lem}\label{lem0}
If $f$ has equi-dimensional fibers, the perverse filtration on $H^m(X,\BQ)$ terminates at $P_mH^m(X, \BQ)$, \emph{i.e.}
\[
P_mH^m(X, \BQ) = H^m(X, \BQ).
\]
In particular, we have $1 \in P_0H^0(X, \BQ)$.
\end{lem}

\begin{proof}
By definition and the decomposition theorem \cite{BBD}, the dimensions of the graded pieces of the perverse filtration are given by
\begin{equation}\label{llm0}
\mathrm{dim} \mathrm{Gr}^P_iH^{m}(X, \BQ) = \mathrm{dim} H^{m-i-(a-r)}\left(Y,{^\mathfrak{p}\CH^i(Rf_* \BQ_X[a-r])}\right).
\end{equation}
Since a perverse sheaf, as a complex of constructible sheaves, is concentrated in degrees $[-b, 0]$, it only has non-trivial cohomology in degrees $\geq -b$. In particular, the right-hand side of (\ref{llm0}) is non-trivial only if 
\[
m-i-(a-r) \geq -b.
\]
Since $r=a-b$, this inequality is equivalent to $m \geq i$.
\end{proof}

A cohomology class $\gamma \in H^l(X, \BQ)$ can be viewed as a morphism $\gamma: \BQ_X \to \BQ_X[l]$, which naturally induces 
\begin{equation}\label{induced}
\gamma: Rf_* \BQ_X \to Rf_* \BQ_X [l]
\end{equation}
after pushing forward along $f$. For an integer $c\geq 0$, we say that $\gamma \in H^l(X, \BQ)$ has \emph{strong perversity} $c$ with respect to $f$ if its induced morphism (\ref{induced}) satisfies
\begin{equation}\label{assumption}
\gamma \left({^\mathbf{p}\tau_{\leq i} Rf_* \BQ_X}\right) \subset {^\mathbf{p}\tau_{\leq i+(c-l)}} \left(Rf_* \BQ_X  [l] \right), \quad \forall i;
\end{equation}
more precisely, the condition (\ref{assumption}) says that the composition
\[
{^\mathbf{p}\tau_{\leq i} Rf_* \BQ_X} \hookrightarrow Rf_* \BQ_X \xrightarrow{\gamma} Rf_* \BQ_X [l]
\]
factors through ${^\mathbf{p}\tau_{\leq i+(c-l)}} \left(Rf_* \BQ_X  [l] \right) \hookrightarrow Rf_* \BQ_X  [l]$ for any $i$.  Notice that $\gamma$ automatically has strong perversity $l$, so this condition is interesting only when $c < l$. Combining Lemma \ref{lem0} and the following lemma, we see that if $f$ has equi-dimensional fibers and $\gamma$ has strong perversity $c$, then $\gamma \in P_cH^*(X, \BQ)$.

\begin{lem}\label{lem1.1}
If $\gamma\in H^l(X, \BQ)$ has strong perversity $c$ with respect to $f$, then taking cup-product with $\gamma$ satisfies
\[
\gamma\cup - : H^m(X, \BQ) \to H^{m+l}(X, \BQ), \quad P_iH^m(X, \BQ) \mapsto P_{i+c}H^{m+l}(X, \BQ).
\]
\end{lem}

\begin{proof}
This follows from taking global cohomology for (\ref{assumption}) and noticing that
\begin{align*}
    H^{m-a+r}\left(Y, {^\mathbf{p}\tau_{\leq i+(c-l)}} \left(Rf_* \BQ_X  [l+a-r] \right)\right) & \\ = H^{(m-a+r)+l}&\left(Y, {^\mathbf{p}\tau_{\leq i+c}} \left(Rf_* \BQ_X  [a-r] \right)\right). \qedhere
    \end{align*}
\end{proof}

In general, the perverse filtration $P_\bullet H^*(X, \BQ)$ may not be multiplicative, \emph{i.e.}, for $\gamma_j \in P_{c_j}H^*(X, \BQ)$ (j=1,2,\dots,s), it may not be true that
\[
\gamma_1 \cup \gamma_2 \cup \cdots \cup \gamma_s \in P_{c_1+\cdots+c_s}H^*(X, \BQ);
\]
see \cite[Exercise 5.6.8]{Park}. In fact, it was proven in \cite[Theorem 0.6]{dCMS} that Conjecture \ref{conj} is equivalent to the multiplicativity of the perverse filtration associated with the Hitchin system. The following easy observation illustrates the advantage of considering strong perversity in view of the multiplicativity issue.

\begin{lem}\label{lem1.2}
If the class $\gamma_j \in H^{l_j}(X, \BQ)$ ($j=1,2,\dots, s$) has strong perversity $c_j$ with respect to $f$, then the cup product 
\[
\gamma_1 \cup \gamma_2 \cup \cdots \cup \gamma_s \in H^{l_1+\cdots+l_s}(X, \BQ)
\]
has strong perversity $\sum_jc_j$.
\end{lem}

\subsection{Vanishing cycles}\label{Sec1.2}

Throughout section \ref{Sec1.2}, we let $g: X \to \BA^1$ be a morphism such that $X$ is nonsingular and irreducible with $X_0 = g^{-1}(0)$ the closed fiber over $0 \in \BA^1$. We consider the vanishing cycle functor 
\[
\varphi_g: D^b_c(X) \rightarrow D^b_c(X_0)
\]
which preserves the perverse $t$-structures,
\[
\varphi_g: \mathrm{Perv}(X) \rightarrow \mathrm{Perv}(X_0).
\]
Here $\mathrm{Perv}(-)$ stands for the abelian category of perverse sheaves. We denote by 
\[
\varphi_g : = \varphi_g(\mathrm{IC}_X) = \varphi_g(\BQ_X[\mathrm{dim}X]) \in \mathrm{Perv}(X_0) \]
the perverse sheaf of vanishing cycles. We use $X' \subset X$ to denote the support of the vanishing cycle complex $\varphi_g$ so that $\varphi_g \in \mathrm{Perv}(X')$.

Recall that for any bounded constructible object $\CK \in D^b_c(X)$, a class $\gamma \in H^l(X, \BQ)$ induces a morphism
\[
\gamma: \CK \to \CK [l] 
\]
via taking the tensor product with $\gamma: \BQ_X \to \BQ_X[l]$. The following lemma shows the compatibility between the vanishing cycle functor and restriction of cohomology classes.

\begin{lem}\label{lem1.3}
With the same notation as above, let $i: X'\hookrightarrow X$ be the closed embedding. The morphism \[
i^*\gamma: \varphi_g \rightarrow \varphi_g[l] \in D^b_c(X')
\]
induced by the class $i^*\gamma$ (applied to the object $\CK = \varphi_g$) coincides with 
\[
\varphi_g(\gamma): \varphi_g \to \varphi_g[l] \in D_c^b(X')
\]
obtained by applying the functor $\varphi_g$ to $\gamma: \BQ_X \to \BQ_X[l]$.
\end{lem}

\begin{proof}
Let $\iota: X_0 \hookrightarrow X$ be the closed embedding of the closed fiber over $0$. By \cite[Definition 8.6.2]{KS}, the vanishing cycle functor can be written as
\begin{equation}\label{vanishing}
\varphi_g(-) = \iota^* R\CH{om}(\CC, -) \in D^b_c(X_0)
\end{equation}
with $\CC$ a fixed complex of sheaves. The morphism obtained
by applying $R\CH{om}(\CC, -)$ to $\gamma: \BQ_X \to \BQ_X[l]$ 
is equivalent to the morphism induced by applying $\gamma$ to $R\CH{om}(\CC, \BQ_X)$.
Similarly, the functor $\iota^*$ sends the morphism induced by $\gamma$
to the morphism induced by $\iota^*\gamma$.

Therefore, if we denote by $\iota': X' \hookrightarrow X_0$ the closed embedding and view $\varphi_g$ as a perverse sheaf on $X'$, we have an equivalent morphism
\begin{equation}\label{321} 
\varphi_g(\gamma) = \iota^*\gamma:  \iota'_*\varphi_g \to \iota'_*\varphi_g[l] \in D^b_c(X_0). 
\end{equation}

Finally, after applying $\iota'^*: D^b_c(X_0) \rightarrow D^b_c(X')$ to (\ref{321}) and noticing $\iota'^*\iota'_*= \mathrm{id}$, we obtain that the class $i^*\gamma = \iota'^*\iota^*\gamma$ induces
\[
\varphi_g(\gamma): \varphi_g \to \varphi_g[l] \in D_c^b(X').
\]
This completes the proof.
\end{proof}

\begin{prop}\label{prop1.4}
Let $g: X \to \BA^1$ and $X'$ be as above. Assume that $X'$ is nonsingular and
\begin{equation}\label{4.1_0}
\varphi_g \simeq \mathrm{IC}_{X'} = \BQ_{X'}[\mathrm{dim}X'] \in \mathrm{Perv}(X').
\end{equation}
Assume further that we have the commutative diagram
\begin{equation*}
\begin{tikzcd}
X' \arrow[r, "i"] \arrow[d, "f'"]
& X \arrow[d, "f"] \\
Y' \arrow[r]
& Y
\end{tikzcd}
\end{equation*}
such that $f$ is proper and $g =   f\circ \nu$ with $\nu: Y \to \BA^1$. Then if a class $\gamma \in H^l(X, \BQ)$ has strong perversity $c$ with respect to $f$, its restriction $i^*\gamma \in H^l(X', \BQ)$ has strong perversity $c$ with respect to $f'$. 
\end{prop}

\begin{proof}
By definition, the morphism 
\begin{equation}\label{1.4_1}
\gamma: Rf_* \BQ_X \to Rf_* \BQ_X[l] \in D^b_c(Y)
\end{equation}
induced by $\gamma$ satisfies
\begin{equation}\label{1.4_2}
\gamma \left({^\mathbf{p}\tau_{\leq i} Rf_* \BQ_X}\right) \subset {^\mathbf{p}\tau_{\leq i+(c-l)}} (Rf_* \BQ_X[l]), \quad \forall i.
\end{equation}
Now we apply the vanishing cycle functor $\varphi_\nu$ to (\ref{1.4_1}). On one hand, we have the base change $Rf'_* \circ \varphi_g \simeq \varphi_\nu\circ Rf_*$ and the fact that the vanishing cycle functor preserves the perverse $t$-structures. Therefore (\ref{1.4_2}) implies  that the morphism
\[
\varphi_g(\gamma): Rf'_* \varphi_g \to Rf'_* \varphi_g[l] \in D^b_c(Y')
\]
satisfies
\[
\varphi_g(\gamma) \left({^\mathbf{p}\tau_{\leq i}} Rf'_* \varphi_g\right) \subset {^\mathbf{p}\tau_{\leq i+(c-l)}} (Rf'_* \varphi_g[l]), \quad \forall i.
\]
On the other hand, using the isomorphism (\ref{4.1_0}) and Lemma \ref{lem1.3}, the above equation means precisely that the morphism $i^*\gamma : \BQ_{X'} \to \BQ_{X'}[l]$ satisfies
\begin{equation*}
(i^*\gamma) \left({^\mathbf{p}\tau_{\leq i} Rf'_* \BQ_{X'}}\right) \subset {^\mathbf{p}\tau_{\leq i+(c-l)}} (Rf'_* \BQ_{X'}[l]), \quad \forall i;
\end{equation*}
that is, the class $i^*\gamma$ has strong perversity $c$ with respect to $f'$.
\end{proof}

\section{Strong perversity for Chern classes}

In this section, we fix the rank $n$ and the degree $d$ with $(n,d)=1$. In Theorem \ref{conj2.7}, we rephrase and then enhance Conjecture \ref{conj} to a statement involving $\CL$-twisted Hitchin systems and strong perversity of Chern classes. It will be proven in Sections 3 and 4.

\subsection{Tautological classes}

As discussed in \cite[Section 0.3]{dCMS}, the $P=W$ conjecture for $\mathrm{GL}_n$ can be reduced to a statement involving tautological classes on $M_{\mathrm{Dol}}$ and the perverse filtration associated with $h: M_{\mathrm{Dol}} \to A$, without reference to the Betti moduli space $M_B$. In this subsection, we recall this reduction step.

For convenience, we work with the $\mathrm{PGL}_n$ Dolbeault moduli space to avoid normalization of a universal family as in \cite{dCMS}; we refer to \cite{Shende} for a detailed discussion concerning the formulation of the $P=W$ conjecture in terms of tautological classes for the $\mathrm{PGL}_n$ Dolbeault moduli space.

Fix $\CN \in \mathrm{Pic}^d(C)$. Let $\widecheck{M}_{\mathrm{Dol}}$ be 
the moduli stack of stable Higgs bundles $(\CE, \theta)$ with $\mathrm{det}(\CE) \simeq \CN$ and $\mathrm{trace}(\theta) = 0$, rigidified with respect to the generic $\mu_n$-stabilizer; this is the same as taking its coarse moduli space.  We refer to this (nonsingular) variety as the $\mathrm{SL}_n$ Dolbeault moduli space of degree $d$. The finite group $\Gamma: = \mathrm{Pic}^0(C)[n]$ acts naturally on $\widecheck{M}_{\mathrm{Dol}}$ via tensor product. The $\mathrm{PGL}_n$ Dolbeault moduli space of degree $d$ is recovered by taking the quotient stack
\begin{equation}\label{233}
\widehat{M}_{\mathrm{Dol}} : = \widecheck{M}_{\mathrm{Dol}}/\Gamma
\end{equation}
which is a nonsingular Deligne--Mumford stack. The $\mathrm{PGL}_n$ Hitchin system 
\[
\widehat{h}: \widehat{M}_{\mathrm{Dol}} \to \widehat{A}: = \bigoplus_{i=2}^n H^0(C, {\Omega_C^1}^{\otimes i})
\]
is induced by the Hitchin map associated with $\widecheck{M}_{\mathrm{Dol}}$ as the $\Gamma$-action is fiberwise with respect to $h$. Analogous to the $\mathrm{GL}_n$  case, we have the perverse filtration $P_*H^*(\widehat{M}_{\mathrm{Dol}}, \BQ)$ associated with $\widehat{h}$. The universal $\mathrm{PGL}_n$-bundle $\CU$ on $C \times \widehat{M}_{\mathrm{Dol}}$ induces Chern characters
\[
\mathrm{ch}_k(\CU) \in H^{2k}(C \times \widehat{M}_{\mathrm{Dol}}, \BQ),\quad k \geq 2. 
\]
The tautological classes $c_k(\gamma)$ are defined to be
\[
c_k(\gamma): = \int_\gamma \mathrm{ch}_k(\CU) = q_{M*}(q_C^*\gamma \cup \mathrm{ch}_k(\CU)) \in H^*(\widehat{M}_{\mathrm{Dol}}, \BQ), \quad \gamma \in H^*(C,\BQ),
\]
where $q_{(-)}$ are the projections from $C \times \widehat{M}_{\mathrm{Dol}}$.

Now we consider the $\mathrm{PGL}_n$ Betti moduli space of degree $d$ with the isomorphism on cohomology provided by non-abelian Hodge theory (\emph{c.f.}\cite[Theorem 1.2.4]{dCHM1}):
\begin{equation}\label{NAH2}
H^*(\widehat{M}_{\mathrm{Dol}}, \BQ) = H^*(\widehat{M}_{B}, \BQ).
\end{equation}
Define the Hodge sub-vector space
\[
^k\mathrm{Hdg}^m(\widehat{M}_B): = W_{2k}H^m(M_B, \BQ) \cap F^kH^m(M_B, \BC) \subset H^m(\widehat{M}_{B}, \BQ).
\]

The following theorem, collecting results of Markman and Shende, provides a complete description of $H^*(\widehat{M}_B, \BQ)$ in terms of the Chern classes $\mathrm{ch}_k(\CU)$ and the weight filtration.

\begin{thm}[\cite{Markman, Shende}]\label{thm2.1}
We use the same notation as above.
\begin{enumerate}
    \item[(i)] The tautological classes $c_k(\gamma)\in H^*(\widehat{M}_{\mathrm{Dol}}, \BQ)$ generate $H^*(\widehat{M}_\mathrm{Dol}, \BQ)$ as a $\BQ$-algebra.
    \item[(ii)] The class $c_k(\gamma)$, passing through the non-abelian Hodge correspondence (\ref{NAH2}), lies in $^k\mathrm{Hdg}^*(\widehat{M}_B)$. In particular, we have a canonical decomposition
    \[
H^*(\widehat{M}_B, \BQ) = \bigoplus_{m,k} {^k\mathrm{Hdg}^m(\widehat{M}_B)}.    \]
\end{enumerate}
\end{thm}

\begin{proof}
The first part was proven in \cite{Markman}, and the second part was proven in \cite{Shende}.
\end{proof}

Theorem \ref{thm2.1} (ii) yields immediately that 
\[
W_{2k}H^m(\widehat{M}_B, \BQ) = W_{2k+1}H^m(\widehat{M}_B, \BQ).
\]
Moreover, by Theorem \ref{thm2.1}, the $P=W$ conjecture implies that each class $\prod_{i=1}^s c_{k_i}(\gamma_i)$ lies in the perverse piece $P_{\Sigma_i{k_i}}H^*(\widehat{M}_{\mathrm{Dol}}, \BQ)$; the latter is in fact equivalent to the $P=W$ conjecture. Indeed, suppose we have 
\begin{equation}\label{taut}
\prod_{i=1}^s c_{k_i}(\gamma_i) \in P_{\Sigma_i{k_i}}H^*(\widehat{M}_{\mathrm{Dol}}, \BQ)
\end{equation}
for any product of tautological classes.
Then we know that $W_{2k}H^*(\widehat{M}_B, \BQ) \subset P_kH^*(\widehat{M}_{\mathrm{Dol}}, \BQ)$. The curious hard Lefschetz theorem proven by Mellit \cite{Mellit} forces the two filtrations $P_\bullet$ and $W_{2\bullet}$ to coincide as long as one contains the other. This was mentioned in the last paragraph of \cite[Section 1]{Mellit}; we include its proof here for the reader's convenience.

\begin{lem}
We denote by $V$ the $\BQ$-vector space (\ref{NAH2}) with the perverse and the weight filtrations $P_\bullet$ and $W_{2\bullet}$. If $W_{2k} \subset P_k$ for all $k$, then $W_{2k} = P_k$ for all $k$.
\end{lem}

\begin{proof}
Assume $r = \mathrm{dim}\widehat{M}$; both filtrations terminate at the $r$-th pieces, \emph{i.e.}, $W_{2r} = P_{r} = V$. We first show that 
\[
W_0 = P_0, \quad W_{2(r-1)} = P_{r-1}. 
\]
In fact, since $W_\bullet  \subset P_\bullet$, we have
\[
\mathrm{dim} W_0 \leq \mathrm{dim} P_0 = \mathrm{dim}V/P_{r-1} \leq \mathrm{dim} V/W_{2(r-1)};
\]
moreover, by the curious hard Lefschetz theorem, each inequality has to be an equality. So our claim follows. 

Then we proceed by applying the same argument to $W_1, P_1$ and $W_{2(r-2)}, P_{r-1}$. The lemma follows by a simple induction.
\end{proof}

In conclusion, we have reduced the $P=W$ conjecture to the following:

\begin{conj}[Equivalent version of P=W]\label{conj2.2}
Condition (\ref{taut}) holds for all products of tautological classes.
\end{conj}

\subsection{Strong perversity for Chern classes for $\CL$-twisted Hitchin systems}

For our purposes, it is important to consider Dolbeault moduli spaces of Higgs bundles, twisted by an effective line bundle $\CL$ (\emph{i.e.}, $H^0(C, \CL) \neq 0$). These moduli spaces have already appeared in \cite{CL, HT2, HT, Yun1, Yun2}; we review the construction here briefly. 

Set $\Omega_\CL$ to be the line bundle $\Omega^1_C \otimes \CL$ on the curve $C$. We denote by $M^\CL_{\mathrm{Dol}}$ the moduli space of stable twisted Higgs bundles
\[
(\CE, \theta), \quad \theta: \CE \to \CE \otimes\Omega_\CL, \quad \mathrm{rk}(\CE) = n, ~~\mathrm{deg}(\CE)=d,
\]
with respect to the slope stability condition. The corresponding Hitchin map
\[
h^\CL: M^\CL_{\mathrm{Dol}} \to A^\CL: = \bigoplus_{i=1}^n H^0\left(C, {\Omega_\CL}^{\otimes i} \right) ,\quad (\CE, \theta) \mapsto \mathrm{char.polynomial}(\theta),
\]
is still proper as in the untwisted case; but it fails to be a Lagrangian fibration when $\mathrm{deg}(\CL)>0$. The $\CL$-twisted $\mathrm{SL}_n$ and $\mathrm{PGL}_n$ Dolbeault moduli spaces $\widecheck{M}^\CL_{\mathrm{Dol}}$ and $\widehat{M}^\CL_{\mathrm{Dol}}$ can be constructed similarly. The moduli space
\[
\widecheck{M}^\CL_{\mathrm{Dol}}:= \{(\CE, \theta) \in M^\CL_{\mathrm{Dol}}|~~ \mathrm{det}(\CE) \simeq \CN \in \mathrm{Pic}^d(C), \mathrm{trace}(\theta) = 0\}
\]
admits a Hitchin map
\[
\widecheck{h}^\CL: \widecheck{M}^\CL_{\mathrm{Dol}} \to \widehat{A}^\CL:= \bigoplus_{i=2}^nH^0\left(C, {\Omega_\CL}^{\otimes i} \right)\]
and a fiberwise $\Gamma = \mathrm{Pic}^0(C)[n]$ action by tensor product. Taking the $\Gamma$-quotient recovers the $\CL$-twisted $\mathrm{PGL}_n$ Hitchin map
\[
\widehat{h}^\CL: \widehat{M}^\CL_{\mathrm{Dol}} =\widecheck{M}^\CL_{\mathrm{Dol}}/\Gamma \to \widehat{A}^\CL. 
\]

An observation in \cite[Section 4]{MS} is that, for a fixed closed point $p\in C$, the $\CL$-twisted and $\CL(p)$-twisted $\mathrm{SL}_n$ Dolbeault moduli spaces can be related via critical loci and vanishing cycles, which we recall in the following.

By viewing a $\CL$-twisted Higgs bundle naturally as a $\CL(p)$-twisted Higgs bundle, we have the natural embedding $i: \widecheck{M}^\CL_{\mathrm{Dol}} \hookrightarrow \widecheck{M}^{\CL(p)}_{\mathrm{Dol}}$ which induces the commutative diagram
\begin{equation}\label{2.3_1}
\begin{tikzcd}
\widecheck{M}^\CL_{\mathrm{Dol}}\arrow[r, hook,"i"] \arrow[d, "\widecheck{h}^\CL"]
&\widecheck{M}^{\CL(p)}_{\mathrm{Dol}} \arrow[d, "\widecheck{h}^{\CL(p)}"] \\
\widehat{A}^\CL \arrow[r,hook]
& \widehat{A}^{\CL(p)}.
\end{tikzcd}
\end{equation}

We recall the following theorem from \cite{MS}.

\begin{thm}[\cite{MS} Theorem 4.5]\label{thm2.5}
There exists a regular function $g: \widecheck{M}^{\CL(p)}_{\mathrm{Dol}} \to \BA^1$ factorized as $g = \nu \circ \widecheck{h}^{\CL(p)}$ with $\nu: \widehat{A}^{\CL(p)} \to \BA^1$ such that
\[
\varphi_g \simeq \mathrm{IC}_{\widecheck{M}^{\CL}_{\mathrm{Dol}}} = \BQ_{\widecheck{M}^{\CL}_{\mathrm{Dol}}}[\mathrm{dim} \widecheck{M}^{\CL}_{\mathrm{Dol}}].
\]
\end{thm}

Since the embedding $i: \widecheck{M}^\CL_{\mathrm{Dol}} \hookrightarrow \widecheck{M}^{\CL(p)}_{\mathrm{Dol}}$ is $\Gamma$-equivariant, taking $\Gamma$-quotients in the diagram (\ref{2.3_1}) yields the following commutative diagram:
\begin{equation}\label{diagram3}
\begin{tikzcd}
\widehat{M}^\CL_{\mathrm{Dol}}\arrow[r, hook, "\widehat{i}"] \arrow[d, "\widehat{h}^\CL"]
&\widehat{M}^{\CL(p)}_{\mathrm{Dol}} \arrow[d, "\widehat{h}^{\CL(p)}"] \\
\widehat{A}^\CL \arrow[r, hook]
& \widehat{A}^{\CL(p)}.
\end{tikzcd}
\end{equation}

\begin{prop}\label{prop2.6}
The vanishing cycle complex $\varphi_{\widehat{g}}$ associated with the regular function
\[
\widehat{g}: = \nu \circ \widehat{h}^{\CL(p)}: \widehat{M}^{\CL(p)}_{\mathrm{Dol}} \to \BA^1
\]
satisfies
\[
\varphi_{\widehat{g}} \simeq \mathrm{IC}_{\widehat{M}^{\CL}_{\mathrm{Dol}}} = \BQ_{\widehat{M}^{\CL}_{\mathrm{Dol}}}\left[\mathrm{dim} \widehat{M}^{\CL}_{\mathrm{Dol}}\right].\]
\end{prop}

\begin{proof}
We reduce Proposition \ref{prop2.6} to Theorem \ref{thm2.5}. Consider the $\Gamma$-quotient map $r: \widecheck{M}^{\CL(p)}_{\mathrm{Dol}} \to \widehat{M}^{\CL(p)}_{\mathrm{Dol}}$. The direct image $r_* \BQ_{\widecheck{M}^{\CL(p)}_{\mathrm{Dol}}}$ admits a natural $\Gamma$-equivariant structure, whose invariant part recovers 
\begin{equation}\label{2.6_1}
\left(r_* \BQ_{\widecheck{M}^{\CL(p)}_{\mathrm{Dol}}}\right)^{\Gamma} \simeq \BQ_{\widehat{M}^{\CL(p)}_{\mathrm{Dol}}}.
\end{equation}
Since $g = \widehat{g}\circ r$, we have
\begin{align*}
\varphi_{\widehat{g}} & = \varphi_{\widehat{g}}\left(\BQ_{\widehat{M}^{\CL(p)}_{\mathrm{Dol}}}\left[\mathrm{dim} \widehat{M}^{\CL(p)}_{\mathrm{Dol}}\right]\right) \\ & \simeq \varphi_{\widehat{g}}\left(r_* \BQ_{\widecheck{M}^{\CL(p)}_{\mathrm{Dol}}}\left[\mathrm{dim} \widecheck{M}^{\CL(p)}_{\mathrm{Dol}}\right]\right)^{\Gamma}\\
& \simeq (r_* \varphi_g)^\Gamma \\ & \simeq  \BQ_{\widehat{M}^{\CL}_{\mathrm{Dol}}}\left[\mathrm{dim} \widehat{M}^{\CL}_{\mathrm{Dol}}\right].
\end{align*}
Here the first equation follows by definition, the second uses (\ref{2.6_1}), the third follows from the base change, and the last is given by Theorem \ref{thm2.5}.
\end{proof}


Now we formulate a sheaf-theoretic enhancement of Conjecture \ref{conj2.2}.

\begin{thm}\label{conj2.7}
There exists an effective line bundle $\CL$ such that the class
\[
\mathrm{ch}_k(\CU^\CL) \in H^{2k}(C\times \widehat{M}^\CL_{\mathrm{Dol}}, \BQ)
\]
has strong perversity $k$ with respect to 
\[
\mathbf{h}^\CL: = \mathrm{id} \times \widehat{h}^\CL:C\times \widehat{M}^\CL_{\mathrm{Dol}} \to C\times \widehat{A}^\CL. 
\]
Here $\CU^\CL$ is the universal $\mathrm{PGL}_n$-bundle over $C\times \widehat{M}^\CL_{\mathrm{Dol}}$.
\end{thm}

Notice that the perversity bound here is stronger than the trivial one of $2k$.  We prove in the rest of this section that Theorem \ref{conj2.7} indeed implies Conjecture \ref{conj2.2}; therefore it implies Conjecture \ref{conj}. We prove Theorem \ref{conj2.7} in Sections 3 and 4.

\subsection{Theorem \ref{conj2.7} implies Conjecture \ref{conj2.2}}\label{Sec2.3} We start with the following claim.


\medskip
\noindent {\bf Claim.} For a point $p\in C$, if the class
\[
\mathrm{ch}_k(\CU^{\CL(p)}) \in H^{2k}(C\times \widehat{M}^{\CL(p)}_{\mathrm{Dol}}, \BQ) 
\]
has strong perversity $k$ with respect to $\mathbf{h}^{\CL(p)}: C\times \widehat{M}^{\CL(p)}_{\mathrm{Dol}} \to C\times \widehat{A}^{\CL(p)}$, then the class
\[
\mathrm{ch}_k(\CU^{\CL}) \in H^{2k}(C\times \widehat{M}^{\CL}_{\mathrm{Dol}}, \BQ)
\]
has strong perversity $k$ with respect to $\mathbf{h}^{\CL}: C\times \widehat{M}^\CL_{\mathrm{Dol}} \to C\times \widehat{A}^\CL$.

\begin{proof}[Proof of Claim]

We consider the commutative diagram obtained from (\ref{diagram3}):

\begin{equation}\label{diagram4}
\begin{tikzcd}
C \times \widehat{M}^\CL_{\mathrm{Dol}}\arrow[r,hook, "\mathbf{i}"] \arrow[d, "\mathbf{h}^\CL"]
&C\times \widehat{M}^{\CL(p)}_{\mathrm{Dol}} \arrow[d, "\mathbf{h}^{\CL(p)}"] \\
C\times \widehat{A}^\CL \arrow[r,hook]
& C\times \widehat{A}^{\CL(p)}.
\end{tikzcd}
\end{equation}
Here $\mathbf{i} = \mathrm{id} \times \widehat{i}$. By definition,  we have $\mathbf{i}^* \CU^{\CL(p)} = \CU^{\CL}$. Therefore 
\[
\mathbf{i}^* \mathrm{ch}_k(\CU^{\CL(p)}) = \mathrm{ch}_k(\CU^\CL).
\]
We know from Proposition \ref{prop2.6} that the vanishing cycle complex associated with the regular function
\[
\mathbf{g}: =  \widehat{g} \circ \mathrm{pr}_M: C\times \widehat{M}^{\CL(p)}_{\mathrm{Dol}} \to\widehat{M}^{\CL(p)}_{\mathrm{Dol}}\to \BA^1
\]
satisfies
\[
\varphi_{\mathbf{g}} \simeq \mathrm{IC}_{C\times \widehat{M}^{\CL}_{\mathrm{Dol}}} = {\BQ}_{C\times \widehat{M}^{\CL}_{\mathrm{Dol}}}\left[\mathrm{dim}~ (C\times \widehat{M}^{\CL}_{\mathrm{Dol}})\right].
\]
Our claim then follows from Proposition \ref{prop1.4} (applied to the diagram (\ref{diagram4}), the class $\mathrm{ch}_k(\CU^{\CL(p)})$, and the map $\mathbf{g}$).
\end{proof}



As a consequence of the claim, the statement of Theorem \ref{conj2.7} holds for $\CL \simeq \CO_C$, \emph{i.e.}, the class 
\[
\mathrm{ch}_k(\CU) \in H^{2k}(C\times \widehat{M}_{\mathrm{Dol}}, \BQ)
\] 
has strong perversity $k$ with respect to 
\[
\mathbf{h}:= \mathrm{id} \times \widehat{h}: C\times \widehat{M}_{\mathrm{Dol}} \to C\times \widehat{A}.
\]


Finally, to deduce the cohomological statement, we note that since $\mathbf{h}$ is the identity map restricting to $C$, its induced perverse filtration can be described as:
\begin{equation}\label{2.2_1}
P_kH^*(C \times \widehat{M}_{\mathrm{Dol}}, \BQ) = H^*(C, \BQ) \otimes P_kH^*(\widehat{M}_{\mathrm{Dol}}, \BQ).
\end{equation}
Choose a homogeneous basis of $H^*(C, \BQ)$:
\[
\Sigma=\{\sigma_0, \sigma_1, \dots, \sigma_{2g+2}\}
\]
with Poincar\'e-dual basis $\{\sigma^{\vee}_{0}, \dots, \sigma^\vee_{2g+2}\}$.

We may express $\mathrm{ch}_k(\CU)$ as
\begin{equation}\label{2.4_1}
\mathrm{ch}_k(\CU) = \sum_{\sigma\in \Sigma} \sigma^\vee \otimes c_k(\sigma) \in H^*(C, \BQ) \otimes H^*(\widehat{M}_{\mathrm{Dol}}, \BQ).
\end{equation}
Since $\mathrm{ch}_k(\CU)$ has strong perversity $k$, Lemma \ref{lem1.1} implies that
\[
\mathrm{ch}_k(\CU) \cup - : P_sH^*(C\times \widehat{M}_{\mathrm{Dol}}, \BQ) \to P_{s+k}H^{*}(C\times \widehat{M}_{\mathrm{Dol}}, \BQ).
\]
Applying this operator to $\sigma \otimes P_sH^*(\widehat{M}_{\mathrm{Dol}}, \BQ) \subset H^*(C\times \widehat{M}_{\mathrm{Dol}}, \BQ)$ with $\sigma \in \Sigma$, we obtain from (\ref{2.2_1}) and (\ref{2.4_1}) that
\[
c_k(\sigma) \cup -: P_sH^*(\widehat{M}_{\mathrm{Dol}}, \BQ) \to P_{s+k}H^*(\widehat{M}_{\mathrm{Dol}}, \BQ). 
\]
In particular, for any class $\gamma \in H^*(C, \BQ)$ we have
\[
c_k(\gamma): P_sH^*(\widehat{M}_{\mathrm{Dol}}, \BQ) \to P_{s+k}H^*(\widehat{M}_{\mathrm{Dol}}, \BQ)\]
which further yields
\[
\prod_ic_{k_1}(\gamma_i) = \left( \prod_ic_{k_1}(\gamma_i) \right) \cup 1 \in P_{\Sigma_ik_i}H^*(\widehat{M}_{\mathrm{Dol}}, \BQ).
 \]
 Here we used $1\in P_0H^0(\widehat{M}_{\mathrm{Dol}}, \BQ)$ in the last equation, which is given by Lemma \ref{lem0}.

This completes the proof that Theorem \ref{conj2.7} implies Conjecture \ref{conj2.2}. \qed

\section{Global Springer theory}

In this section, we review Yun's global Springer theory \cite{Yun1, Yun2, Yun3} and use it to deduce Theorem \ref{conj2.7}, assuming a support theorem (Theorem \ref{thm3.2}). Global Springer theory was previously used to study perverse filtrations for affine Springer fibers in terms of Chern classes \cite{OY, OY2}.  In our setting, we require a partial extension of Yun's results from the elliptic locus to the entire Hitchin base.

\subsection{Notations}
We fix $\CL$ to be an effective line bundle of sufficiently large degree
($\mathrm{deg}\left(\CL\right)> 2g$ is enough for us).
 Since from now on we only concern the $\CL$-twisted moduli spaces, we will use $M, \widecheck{M}$, and $\widehat{M}$ to denote the $\CL$-twisted $\mathrm{GL}_n$, $\mathrm{SL}_n$, and $\mathrm{PGL}_n$ Dolbeault moduli spaces $M^\CL_{\mathrm{Dol}}, \widecheck{M}^\CL_{\mathrm{Dol}}$, and $\widehat{M}^\CL_{\mathrm{Dol}}$ respectively. For the same reason, we will uniformly use the term \emph{Higgs bundles} to call $\CL$-twisted Higgs bundles.

From now on we let $G = \mathrm{PGL}_n$, $\mathfrak{g}$ its Lie algebra, $B \subset G$ the Borel subgroup induced by upper triangular matrices with $\mathfrak{b}$ the Lie algebra, $T\subset B$ the maximal torus given by diagonal matrices, and $W \simeq \mathfrak{S}_n$ the Weyl group. We denote by $\BX^*(T)$ the character group of $T$; it is isomorphic to $\BZ^{n-1}$ as an abelian group.

\subsection{Parabolic moduli stacks}

Let $\widehat{\FM}$ be the moduli stack of $G$-Higgs bundles on $C$; we do not impose any stability condition on $\widehat{\FM}$ so that it is only a (singular) algebraic stack. The stable locus of $\widehat{\FM}$ is a nonsingular Deligne--Mumford substack
\[
\widehat{M} \hookrightarrow \widehat{\FM}.
\]

Yun's global Springer theory \cite{Yun1} constructs an algebraic stack $\widehat{\FM}^\mathrm{par}$ over $C \times \widehat{\FM}$:
\begin{equation}\label{PI}
\pi: \widehat{\FM}^\mathrm{par} \to C \times \widehat{\FM}
\end{equation}
which is a global analog of the Grothendieck simultaneous resolution. There are two equivalent constructions of $\widehat{\FM}^{\mathrm{par}}$ given in \cite[Section 2.1]{Yun1}.

The first is to construct $\widehat{\FM}^{\mathrm{par}}$ as the moduli stack of parabolic Higgs bundles, which are quadruples $(x, \CE, \theta, \CE_x^B)$, with $(\CE, \theta)$ is a $G$-Higgs bundle, $x\in C$ a closed point, and $\CE_x^B$ a $B$-reduction of $\CE$ at $x$, satisfying the constraint that $\theta$ is compatible with $\CE_x^B$; see \cite[Definition 2.1.1]{Yun1}. Then the morphism $\pi$ is given by forgetting the $B$-reduction:
\[
\pi (x, \CE, \theta, \CE_x^B)  = (x, (\CE,\theta)) \in C \times \widehat{\FM}.
\]

The second construction is via the Grothendieck simultaneous resolution $\pi_G: [\mathfrak{b}/B] \to [\mathfrak{g}/G]$. More precisely, let $\rho_\CL$ be the $\BG_m$-torsor over $C$ associated with the line bundle $\Omega_\CL$. Denote by $[\mathfrak{g}/G]_\CL$ (resp. $[\mathfrak{b}/B]_\CL$) the family of $[\mathfrak{g}/G]$ (resp. $[\mathfrak{b}/B]$) over the curve $C$ twisted by the torsor $\rho_\CL$. We have a tautological evaluation map
\begin{equation}\label{evaluation}
    \begin{tikzcd}[column sep=small]
    C\times \widehat{\FM} \arrow[dr, ""] \arrow[rr, "\mathrm{ev}"] & & {[\mathfrak{g}/G]_\CL} \arrow[dl, ""] \\
       & C  & 
\end{tikzcd}
\end{equation}
which is a natural $C$-morphism; after base change to a closed point $x\in C$, the map $\mathrm{ev}$ sends a $G$-Higgs bundle to the evaluation of its Higgs field at $x$. The morphism (\ref{PI}) is then induced by the base change of the Grothendieck simultaneous resolution along the evaluation map (see \cite[Lemma 2.1.2]{Yun1}):
\begin{equation*}
\begin{tikzcd}
\widehat{\FM}^{\mathrm{par}} \arrow[r, "\mathrm{ev}^p"] \arrow[d, "\pi"]
&{[\mathfrak{b}/B]_\CL} \arrow[d, "\pi_G"] \\
C\times \widehat{\FM} \arrow[r, "\mathrm{ev}"]
& {[\mathfrak{g}/G]_\CL}.
\end{tikzcd}
\end{equation*}

The parabolic Hitchin system
\[
\widehat{h}^\mathrm{par}: \widehat{\FM}^{\mathrm{par}} \to C\times \widehat{A}
\]
is the composition of $\pi:\widehat{\FM}^{\mathrm{par}} \to C\times \widehat{\FM}$ and the morphism $\mathbf{h}: C\times \widehat{\FM} \to C\times \widehat{A}$
induced by the standard (stacky) Hitchin map $\widehat{h}: \widehat{\FM} \to \widehat{A}$.\footnote{Recall that we always use the $\CL$-twisted version.}

In general the moduli stack $\widehat{\FM}^{\mathrm{par}}$ is singular, and the parabolic Hitchin map is not proper. In the next section we will impose a stability condition and then restrict to the stable locus. 

\subsection{Stable loci}

We define the \emph{stable locus} of the moduli of parabolic $G$-Higgs bundles to be
\[
\widehat{M}^{\mathrm{par}}: = (C \times \widehat{M}) \times_{C \times \widehat{\FM}} \widehat{\FM}^{{\mathrm{par}}};
\]
equivalently, it fits into the Cartesian diagrams
\begin{equation}\label{diag0}
\begin{tikzcd}
\widehat{M}^{\mathrm{par}} \arrow[r,hook] \arrow[d, "\pi"]&  \widehat{\FM}^{\mathrm{par}} \arrow[r, "\mathrm{ev}^p"] \arrow[d, "\pi"]
&{[\mathfrak{b}/B]_\CL} \arrow[d, "\pi_G"] \\ C \times \widehat{M} \arrow[r,hook]
& C\times \widehat{\FM} \arrow[r, "\mathrm{ev}"]
& {[\mathfrak{g}/G]_\CL}.
\end{tikzcd}
\end{equation}

We use the same notation as for the stacky case to denote the parabolic Hitchin map restricted to the stable locus:
\begin{equation}\label{parah}
\widehat{h}^{\mathrm{par}}: \widehat{M}^{\mathrm{par}} \to C\times \widehat{A}.
\end{equation}

\begin{prop}\label{prop3.1}
\begin{enumerate}
    \item[(i)] The moduli stack $\widehat{M}^{\mathrm{par}}$ is nonsingular and Deligne--Mumford; and
    \item[(ii)] the parabolic Hitchin map (\ref{parah}) is proper.
\end{enumerate}
\end{prop}

\begin{proof}
The Deligne--Mumford part of (i) follows from the proof of \cite[Proposition 2.5.1 (3)]{Yun1}. Indeed, the left vertical arrow in the diagram (\ref{diag0}) is the pullback of $\pi_G$, therefore it is schematic and of finite type. This implies that $\widehat{M}^{\mathrm{par}}$ is Deligne--Mumford since $C \times \widehat{M}$ is Deligne--Mumford.

To prove the smoothness part of (i), we use the evaluation map (\ref{evaluation}). Recall that \cite[Proposition 4.1]{MS} (which was proven via deformation theory) shows that the $C$-morphism $\mathrm{ev}$ is smooth after restricting over the stable locus $C \times \widehat{M}$. Hence by the Cartesian diagrams of (\ref{diag0}), the evaluation map
\[
\mathrm{ev}: \widehat{M}^{\mathrm{par}} \to {[\mathfrak{b}/B]_\CL}
\]
is also smooth. Since the target is a nonsingular algebraic stack, so is the source.

(ii) follows directly from (\ref{diag0}) and that $\widehat{h}: \widehat{M} \to \widehat{A}$ is proper.
\end{proof}

As a consequence of Proposition \ref{prop3.1}, the direct image complex
\[
R{\widehat{h}^{\mathrm{par}}}_* ~\BQ_{\widehat{M}^{\mathrm{par}}} \in D^b_c(C\times \widehat{A})
\]
satisfies the decomposition theorem \cite{BBD}.

\begin{thm}[Support theorem for parabolic Hitchin map]\label{thm3.2}
The decomposition for the parabolic Hitchin map $\widehat{h}^{\mathrm{par}}$ has full support, \emph{i.e.}, any non-trivial simple perverse summand of \[
^\mathfrak{p}\CH^i\left(R{\widehat{h}^{\mathrm{par}}}_* ~\BQ_{\widehat{M}^{\mathrm{par}}} \right), \quad  \forall i\in \BZ\]
has support $C\times \widehat{A}$.
\end{thm}

We will postpone the proof of Theorem \ref{thm3.2} to Section 4.

\subsection{Proof of Theorem \ref{conj2.7}}

We first prove Theorem \ref{conj2.7} assuming Theorem \ref{thm3.2}.

There are three ingredients of global Springer theory
\[
\pi: \widehat{M}^\mathrm{par} \to C \times \widehat{M}
\]
which are important in the proof of Theorem \ref{conj2.7}. We summarize them as follows.
\begin{enumerate}
    \item[(A)] \emph{(Splitting of the universal $G$-bundle)} As explained in the second bullet point of \cite[Construction 6.1.4]{Yun1}, each Chern root of $\pi^*\CU$ is of the form $c_1(L(\xi))$ where $\xi$ is an element in $\BX^*(T)$ and $L(\xi)$ is the tautological line bundle on $\widehat{M}^{\mathrm{par}}$ associated with $\xi$. 
    
    More precisely, $\pi^*\CU$ is the universal $G$-bundle on $\widehat{M}^{\mathrm{par}}$ whose fiber over $(x,\CE,\theta, \CE^B_x)$ is $\CE_x$; the $B$-reduction $\CE_x^B$ yields a $T$-torsor over $\widehat{M}^{\mathrm{par}}$ which induces for each $\xi \in \BX^*(T)$ a line bundle $L(\xi)$.

    \item[(B)] \emph{(Strong perversity for $c_1(L(\xi))$)} By \cite[Lemma 3.2.3]{Yun2}, there exists a Zariski dense open subset of $C\times \widehat{A}$ over which the operator
    \[
c_1(L(\xi)): R\widehat{h}^{\mathrm{par}}_*\BQ_{\widehat{M}^{\mathrm{par}}} \to R\widehat{h}^{\mathrm{par}}_*\BQ_{\widehat{M}^{\mathrm{par}}}[2]    \]
has strong perversity $1$ for any $\xi \in \BX^*(T)$. 
Since $c_1(L(\xi))$ automatically has strong perversity $2$, showing it has strong perversity $1$ is equivalent to showing that the induced morphism of perverse cohomology sheaves
\[
{}^{p}\CH^i\left( R\widehat{h}^{\mathrm{par}}_*\BQ_{\widehat{M}^{\mathrm{par}}}\right)\rightarrow
{}^{p}\CH^i\left(R\widehat{h}^{\mathrm{par}}_*\BQ_{\widehat{M}^{\mathrm{par}}}[2]\right)
\]
vanishes for each $i$.  Once we have Theorem \ref{thm3.2}, these sheaves have full support over the entire base, so this vanishing (and thus strong perversity $1$) extends over the total base $C \times \widehat{A}$ as well. 
In fact, \cite[Lemma 3.2.3]{Yun2} was proven in exactly this way using a support theorem \cite[Section 4.6.2]{Yun2} for the elliptic locus.

    \item[(C)] \emph{(Springer's Weyl group action)} By the Cartesian diagrams (\ref{diag0}), we may pullback Springer's sheaf-theoretic Weyl group action from the Grothendieck simultaneous resolution; in particular, the object $R\pi_* \BQ_{\widehat{M}^{\mathrm{par}}}$ admits a canonical $W$-action whose invariant part recovers the trivial local system
    \[
\left(R\pi_* \BQ_{\widehat{M}^{\mathrm{par}}} \right)^W = \BQ_{C\times \widehat{M}}.  \]
By taking global cohomology we have
\[
H^*(\widehat{M}^{\mathrm{par}}, \BQ) = H^*(C\times \widehat{M}, \BQ) \oplus \left(\mathrm{variant~~part~~of~~} W \right).
\]
\end{enumerate}

Now we prove Theorem \ref{conj2.7}. We first prove its parabolic version.

\medskip
\noindent {\bf Claim.} The class
\[
\mathrm{ch}_k\left(\pi^* \CU\right) \in H^{2k}(\widehat{M}^{\mathrm{par}}, \BQ)
\]
has strong perversity $k$ with respect to the parabolic Hitchin system (\ref{parah}).

\begin{proof}[Proof of Claim]
By (A), the Chern character $\mathrm{ch}_k(\pi^*\CU)$ can be expressed in terms of $c_1(L(\xi))$. Hence the claim follows from Lemma \ref{lem1.2} and the strong perversity of $c_1(L(\xi))$ given by (B).
\end{proof}

Next, we reduce Theorem \ref{conj2.7} to the claim above via the following lemma.

\begin{lem}
For a class $\gamma \in H^l(C \times \widehat{M}, \BQ)$, if $\pi^*\gamma \in H^l(\widehat{M}^{\mathrm{par}}, \BQ)$ has strong perversity $k$ with respect to $\widehat{h}^{\mathrm{par}}: \widehat{M}^{\mathrm{par}} \to C\times \widehat{A}$, then $\gamma$ has strong perversity $k$ with respect to ${\mathbf{h}} : C\times \widehat{M} \to C\times \widehat{A}$.
\end{lem}

\begin{proof}
    This is a consequence of (C). We write the class $\gamma$ as a map 
    \begin{equation}\label{3.3_1}
\gamma: \BQ_{C\times \widehat{M}} \to \BQ_{C\times \widehat{M}}[l],   
\end{equation}
whose pullback 
\begin{equation}\label{3.3_2}
\pi^*\gamma: \BQ_{\widehat{M}^{\mathrm{par}}} \to  \BQ_{\widehat{M}^{\mathrm{par}}}[l]
\end{equation}
recovers the class $\pi^*\gamma \in H^l(\widehat{M}^{\mathrm{par}}, \BQ)$. By the projection formula $\pi_*\pi^*\gamma = |W|\cdot \gamma$, we may recover the morphism (\ref{3.3_1}) from (\ref{3.3_2}) by derived pushing forward to $C \times \widehat{M}$ and taking the $W$-invariant part.

Now by the assumption we know that the action of $\pi^*\gamma$ on the object $R\widehat{h}^{\mathrm{par}}_*\BQ_{\widehat{M}^{\mathrm{par}}}$ satisfies
\begin{equation}\label{3.3_4}
\pi^*\gamma: {^\mathfrak{p}\tau_{\leq i}}R\widehat{h}^{\mathrm{par}}_*\BQ_{\widehat{M}^{\mathrm{par}}} \to {^\mathfrak{p}\tau_{\leq i+(k-l)}}\left(R\widehat{h}^{\mathrm{par}}_*\BQ_{\widehat{M}^{\mathrm{par}}}[l]\right).
\end{equation}
Since we have
\[
R\widehat{h}^{\mathrm{par}}_*\BQ_{\widehat{M}^{\mathrm{par}}} = R\mathbf{h}_*\left(R\pi_* \BQ_{\widehat{M}^{\mathrm{par}}}\right),
\] 
the ingredient (C) produces a natural $W$-action on this object whose invariant part recovers $R\mathbf{h}_*\BQ_{C\times \widehat{M}}$; the operator
\begin{equation}\label{3.3_3}
\gamma:  R\mathbf{h}_*\BQ_{C\times \widehat{M}} \to R\mathbf{h}_*\BQ_{C\times \widehat{M}}[l]  
\end{equation}
is then recovered from the $W$-invariant part of the action of $\pi^*\gamma$ on $R\widehat{h}^{\mathrm{par}}_*\BQ_{\widehat{M}^{\mathrm{par}}}$. In particular, the desired property concerning (\ref{3.3_3}) follows from the $W$-invariant part of (\ref{3.3_4}).
\end{proof}

Thus we have completed the proof of Theorem \ref{conj2.7}. \qed

\section{Parabolic support theorem}

In the last section, we prove Theorem \ref{thm3.2}. This is the parabolic version of the Chaudouard--Laumon support theorem \cite{CL}. For the proof, we ultimately reduce it to a relative dimension bound which we establish in Section \ref{sec4.4}.

\subsection{Review of support theorems}\label{Sec4.1}

We start with a review of support theorems for Hitchin systems.

For a proper morphism $f: X\to Y$ with $X,Y$ nonsingular Deligne--Mumford stacks, the decomposition theorem \cite{BBD} yields
\[
Rf_* \BQ_X \simeq \bigoplus_i {^\mathfrak{p}\CH^i(Rf_* \BQ_X)[-i]}
\]
with ${^\mathfrak{p}\CH^i}(Rf_* \BQ_X)$ semisimple perverse sheaves on $Y$; we say that a closed $Z \subset Y$ is a support of $f$ if it is a support of a simple summand of some ${^\mathfrak{p}\CH^i}(Rf_* \BQ_X)$. A particularly interesting case is that $f$ has full support, that is, $Y$ is the only support of $f$; in this case the cohomology of any closed fiber of $f$ is governed by the nonsingular fibers.

The study of supports for Hitchin systems was initiated by B.C. Ng\^o and is crucial in his proof of the fundamental lemma of the Langlands program \cite{Ngo}. He determines all the supports for the Hitchin system (including the $\CL$-twisted cases) after restricting to the elliptic locus of the Hitchin base; that is the subset formed by integral spectral curves.

After Ng\^o's work, Chaudouard--Laumon \cite{CL} observed that, if we consider the moduli space of $\CL$-twisted $\mathrm{GL}_n$ stable Higgs bundles with $\mathrm{deg}(\CL) >0$, then Ng\^o's support theorem can be extended to the total Hitchin base; in particular they showed that the $\CL$-twisted $\mathrm{GL}_n$ Hitchin system has full support. Chaudouard--Laumon's idea was extended to the $\mathrm{SL}_n$ case \cite{dC_SL}, the endoscopic moduli spaces \cite{MS}, and singular cases involving strictly semistable Higgs bundles \cite{MS2, MS3}. See also \cite{dCHM2, MM} concerning the supports over the open subset of reduced spectral curves for the untwisted (\emph{i.e.} $\CL = 0$) Hitchin system.

The idea of Chaudouard--Laumon \cite{CL} is to show that each support of the $\CL$-twisted Hitchin system has generic point lying in the elliptic locus; this is achieved by combining two constraints: (I) the support inequality for a weak abelian fibration, and (II) $\delta$-regularity for spectral curves. We describe them in more detail.
\begin{enumerate}
    \item[(I)] \emph{(Support inequaltiy)} The Hitchin system admits the structure of a weak abelian fibration, that is, there is a commutative group scheme $P$ over the Hitchin base acting on the moduli space which satisfies certain properties. Then an argument generalizing the Goresky--MacPherson inequality leads to a codimension estimate for the supports. More concretely, it says that any support $Z$ has codimension bounded above by the $\delta$-function of the group $P$. This part was already carried out in Ng\^o \cite[Section 7]{Ngo}; see \cite[Section 1]{MS2} for a summary.
    \item[(II)] \emph{($\delta$-regularity)} In the case of type $A$ and for the stable locus, the group scheme is obtained from the multi-degree 0 relative Picard (or, for $\mathrm{SL}_n$, Prym) variety associated with the family of spectral curves. Then we have the Severi inequality, referred to as \emph{$\delta$-regularity}, for the spectral curves. We refer to \cite[Theorem 5.4.4]{dC_SL} for the precise statement.
\end{enumerate}

Using (I) and (II), one can deduce that no support is allowed to have generic point lying outside the open subset of integral curves; this was explained in \cite[Section 11]{CL} for $\mathrm{GL}_n$, in \cite[Section 6.2]{dC_SL} for $\mathrm{SL}_n$, and in \cite[Section 4.5]{MS2} for a general complete linear system in a del Pezzo surface. The argument is to combine the two inequalities above to deduce a numerical contradiction if there is a support that appears outside the elliptic locus.


\subsection{Parabolic Hitchin systems}
Now we focus on the $\CL$-twisted parabolic Hitchin system (\ref{parah}). Let $\widehat{A}^{\mathrm{ell}} \subset \widehat{A}$ be the open subset parameterizing integral spectral curves (the elliptic locus). Following Ng\^o's method, Yun \cite{Yun2, Yun3} proved a parabolic support theorem over $C \times \widehat{A}^{\mathrm{ell}}$ and determined all the supports of (\ref{parah}) in $C \times \widehat{A}^\mathrm{ell}$. More precisely, by \cite[Section 2]{Yun3} any strict subset of $C \times \widehat{A}^{\mathrm{ell}}$, which is a support of $\widehat{h}^{\mathrm{par}}|_{C \times \widehat{A}^{\mathrm{ell}}}$, is a component of the \emph{endoscopic loci} governed by the endoscopic theory of $G$. As we only consider the special case $G = \mathrm{PGL}_n$, there are no nontrivial endoscopic loci and the restricted Hitchin map $\widehat{h}^{\mathrm{par}}|_{C \times \widehat{A}^{\mathrm{ell}}}$ has full support.

To prove Theorem \ref{thm3.2}, it suffices to show that there is no support of (\ref{parah}) whose generic point lying outside $C \times \widehat{A}^{\mathrm{ell}}$.

Since stability for a $\mathrm{PGL}_n$ Higgs bundle is described by its corresponding vector bundle (see (\ref{233})), it is more convenient to work with the $\mathrm{SL}_n$ moduli spaces. We consider the following Cartesian diagram
\begin{equation}\label{diag321}
\begin{tikzcd}
\widecheck{M}^{\mathrm{par}} \arrow[r, "(-)/\Gamma"] \arrow[d, "\pi'"]
& \widehat{M}^{\mathrm{par}} \arrow[d, "\pi"] \\
C\times \widecheck{M} \arrow[r, "(-)/\Gamma"]
& C \times \widehat{M}.
\end{tikzcd}
\end{equation}
Here the horizontal maps are given by the natural quotient maps by $\Gamma = \mathrm{Pic}^0(C)[n]$. To describe the map $\pi'$ at the left-side of the diagram (see \cite[Example 2.2.5]{Yun1}),  we recall that $\widecheck{M}$ parameterizes traceless Higgs bundles with fixed determinant  
\[
(\CE, \theta) , \quad \theta: \CE \to \CE\otimes \Omega_\CL, \quad \mathrm{rk}(\CE) = n, ~~\mathrm{det}(\CE) \simeq \CN,~~ \mathrm{trace}(\theta) = 0
\]
with respect to slope stability. Similarly $\widecheck{M}^{\mathrm{par}}$ parameterizes 
\[
(x, \CE = \CE_0 \supset \CE_1 \supset \cdots \supset \CE_n =  \CE_0(-x), \theta)
\]
where $x\in C$, $(\CE, \theta) \in \widecheck{M}$, each $\CE_i$ in the flag is of rank $n$ with $\CE_i/\CE_{i+1}$ a length 1 skyscraper sheaf supported at $x$, the Higgs field preserves the flag $\theta(\CE_i) \subset \CE_i\otimes \Omega_\CL$, and the map $\pi'$ in the diagram is the forgetful map. We note that the stability condition for a parabolic Higgs bundle is determined by the stability of the underlying Higgs bundle. 

Following \cite[Example 2.2.5]{Yun1} we may also describe $\widecheck{M}^{\mathrm{par}}$ via spectral curves. If we present a point in $\widecheck{M}$ as a 1-dimensional sheaf $\CF_\alpha$ supported on a spectral curve $C_\alpha \subset \mathrm{Tot}(\Omega_\CL)$ with $p_\alpha: C_\alpha \to C$ the projection, then for any $x\in C$, a closed point in the fiber $\pi'^{-1}(x, \CF_\alpha)$ is represented by
\[
\CF_\alpha = \CF_0 \supset \CF_1 \supset \cdots \supset \CF_n = \CF_0\otimes p_\alpha^* \CO_C(-x), \quad \mathrm{lengh}(\CF_i/\CF_{i+1}) = 1;
\]
see \cite[(2.4)]{Yun1}.

Now we consider the $\mathrm{SL}_n$ parabolic Hitchin system
\begin{equation}\label{SLh}
\widecheck{h}^{\mathrm{par}}: \widecheck{M}^{\mathrm{par}} \to C\times \widecheck{M} \to C\times \widehat{A}.
\end{equation}
From the diagram (\ref{diag321}), it recovers the $\mathrm{PGL}_n$ Hitchin system $\widehat{h}^{\mathrm{par}}: \widehat{M}^{\mathrm{par}} \to C \times \widehat{A}$ by taking the quotient of the source by $\Gamma$. In particular, we have
\[
R{\widehat{h}^{\mathrm{par}}}_* ~\BQ_{\widehat{M}^{\mathrm{par}}} = \left(R{\widecheck{h}^{\mathrm{par}}}_* ~\BQ_{\widecheck{M}^{\mathrm{par}}}\right)^\Gamma \in D_c^b(C\times \widehat{A}).
\]
Hence in order to prove Theorem \ref{thm3.2}, it suffices to show that there is no support of (\ref{SLh}) with generic point lying outside $C \times \widehat{A}^{\mathrm{ell}}$. In the next two sections, we adapt the strategy of Chaudouard--Laumon \cite{CL} and de Cataldo \cite{dC_SL} (for the $\mathrm{SL}_n$-version) to this parabolic setting, and verify the parabolic analogs of (I) and (II) of Section \ref{Sec4.1}.

\subsection{Weak abelian fibrations and $\delta$-regularity}

A general approach for proving support theorems was given in \cite[Theorem 1.8]{MS2}. We apply it here for the parabolic Hitchin system (\ref{SLh}). In this section, we explain how existing results allow us to reduce Theorem \ref{thm3.2} to a relative dimension bound (\ref{relative}). Then in the next section we prove this bound. 

For the reader's convenience, we first recall some of the necessary ingredients. As mentioned in (I) of Section \ref{Sec4.1}, the notion of \emph{weak abelian fibration} introduced by Ng\^o \cite{Ngo} plays a key role. A weak abelian fibration is a triple $(\CP, \CM, \CA)$ with morphisms 
\[
f: \CM\to \CA, \quad g: \CP \to \CA,
\]
such that all $\CP, \CM, \CA$ are nonsingular, $f$ is proper, $g: \CP \to \CA$ is a smooth group scheme, and $\CP$ acts on $\CM$ relatively over $\CA$; they satisfy the following conditions:
\begin{enumerate}
    \item[(i)] every closed fiber of the $\CP \to \CA$ has dimension $= \mathrm{dim}\CM - \mathrm{dim}\CA$,
    \item[(ii)] the action of $\CP$ on $\CM$ has affine stabilizers, and
    \item[(iii)] the Tate module associated with $\CP \to \CA$ is polarizable.
\end{enumerate}
We refer to \cite[Section 1.1]{MS2} for more details.

Now we show that (\ref{SLh}) is naturally enhanced into a weak abelian fibration. We first construct the $(C\times \widehat{A})$-group scheme $P$ which acts on $\widecheck{M}^{\mathrm{par}}$. Recall that in \cite{dC_SL} de Cataldo showed that $\widecheck{M}$ admits a weak abelian fibration structure $(\widecheck{P}, \widecheck{M}, \widehat{A})$. Here the $\widehat{A}$-group scheme $\widecheck{P}$ which acts on $\widecheck{M}$ is given by the identity component of the relative Prym variety associated with the spectral curves \cite[(43)]{dC_SL}. For a spectral curve $p_\alpha: C_\alpha \to C$ with a Higgs bundle given by $\CF_\alpha$ supported on $C_\alpha$, the Prym action is induced by tensor product
\[
\CQ \cdot \CF_{\alpha} = \CQ \otimes \CF_\alpha, \quad \CQ \in \mathrm{Prym}^0(C_\alpha/C) = \widecheck{P}_{[C_\alpha]},~ \quad [C_\alpha] \in \widehat{A}. 
\]
We denote by $P$ the $(C\times \widehat{A})$-group scheme obtained by the pullback of $\widecheck{P}$. The $P$-action on $C\times \widecheck{M}$ can be lifted to $\widecheck{M}^{\mathrm{par}}$:
\[
\CQ \cdot (x, \CF_\alpha = \CF_0 \supset \CF_1 \supset \cdots \supset \CF_n) =  (x, \CQ\otimes\CF_\alpha = \CQ \otimes \CF_0 \supset \CQ\otimes\CF_1 \supset \cdots \supset \CQ \otimes \CF_n),
\] 
since the stability condition does not rely on the flag.

\begin{prop}\label{prop4.1}
The triple $(P, \widecheck{M}^{\mathrm{par}}, C\times \widehat{A})$ forms a weak abelian fibration, \emph{i.e.},  it satisfies (i,ii,iii) above or of \cite[Section 1.1]{MS2}.
\end{prop}

\begin{proof}
    (i) only concerns the dimension of $\widecheck{M}$ which is clear from (\ref{diag0}), and (iii) only depends on the group scheme $P$ which follows from the corresponding property for $\widecheck{P}$ proved in \cite[Theorem 4.7.2]{dC_SL}. To prove (ii): since the $P$-action on $\widecheck{M}^{\mathrm{par}}$ is a lifting of the $P$-action on $C\times \widecheck{M}$, and the latter has affine stabilizers by \cite{dC_SL}, we obtain that the $P$-stabilizer for any point $z \in \widecheck{M}^{\mathrm{par}}$ is contained in the (affine) $P$-stabilizer for $\pi(z) \in C \times \widecheck{M}$. This proves (ii).
\end{proof}


Next, we consider the $\delta$-function on the base $C\times \widehat{A}$. The group scheme $P$ endows $C\times \widehat{A}$ with an upper semi-continuous function
\[
\delta: C\times \widehat{A} \to \BN
\]
calculating the dimension of the affine part of the commutative group scheme given by each closed fiber. For a closed subset $Z\subset C\times \widehat{A}$, we define $\delta_Z$ to be the minimal value of the function $\delta$ on $Z$. 

We say that the weak abelian fibration  $(P, \widecheck{M}^{\mathrm{par}}, C\times \widehat{A})$ given by Proposition \ref{prop4.1} satisfies the \emph{support inequality} ((I) of Section \ref{Sec4.1}), if the inequality
\begin{equation}\label{delta_ineq}
\mathrm{codim}_{C\times \widehat{A}}Z \leq \delta_Z
\end{equation}
holds for any irreducible support $Z \subset C\times \widehat{A}$ associated with (\ref{SLh}). Furthermore, by \cite[Theorem 1.8]{MS2}, the support inequality (\ref{delta_ineq}) follows from the relative dimension bound
\begin{equation}\label{relative}
   \tau_{> 2\mathbf{d}}\left( R\widecheck{h}^{\mathrm{par}}_* \BQ_{\widecheck{M}^{\mathrm{par}}}\right) = 0, \quad \mathbf{d}: = \mathrm{dim} \widecheck{M}^{\mathrm{par}} - \mathrm{dim} (C\times \widehat{A})
\end{equation}
with $\tau_{>*}$ the standard truncation functor. 

Once we have (\ref{relative}),
we can combine the support inequality (\ref{delta_ineq}) with $\delta$-regularity of 
$\widecheck{P}$ \cite[corollary 5.4.4]{dC_SL} to prove 
Theorem \ref{thm3.2} as follows.  If $Z \subset C \times \widehat{A}$ is an irreducible support of $\widecheck{h}^{\mathrm{par}}$, the projection $W = \mathrm{pr}_{\widehat{A}}(Z) \subset \hat{A}$ also satisfies the support inequality
\[
\mathrm{codim}_{\widehat{A}} W \leq \mathrm{codim}_{C\times \widehat{A}}Z \leq \delta_Z = \delta_W.
\]
Here we used in the second inequality the support inequality (\ref{delta_ineq}) for $\widecheck{h}^{\mathrm{par}}$, and used in the last equality that the $\delta$-function on $C\times \widehat{A}$ is pulled back from $\widehat{A}$. The argument of \cite[Section 6.2]{dC_SL} then shows that the generic point of $W$ lies in $\widehat{A}^{\mathrm{ell}}$. Hence the generic point of $Z$ lies in $C\times \widehat{A}^{\mathrm{ell}}$ as we desired.


In conclusion, this reduces Theorem \ref{thm3.2} to (\ref{relative}).

\subsection{Proof of the relative dimension bound (\ref{relative})}\label{sec4.4}

We need to show that each closed fiber of the morphism
$$\widecheck{h}^{\mathrm{par}}: \widecheck{M}^{\mathrm{par}} \rightarrow C\times \widehat{A}$$
has dimension
\[
\mathbf{d} = \mathrm{dim}\widecheck{M}^{\mathrm{par}} - \mathrm{dim} (C\times \widehat{A}) =    \frac{n(n-1)}{2}\deg(\CL) + (n^2-1)(g-1).\]
Here the last equation is obtained by the dimension formula of \cite[Proposition 2.4.6]{dC_SL} directly:
\[
\mathbf{d} =  \mathrm{dim}\widecheck{M} - \mathrm{dim}\widehat{A} = \frac{n(n-1)}{2} \deg(\CL) + (n^2-1)(g-1).
\]


First,  we note that the morphism $\widecheck{h}^{\mathrm{par}}$ is surjective since the usual Hitchin map $\widecheck{h}: \widecheck{M} \to \widehat{A}$ is surjective. Furthermore, as $\widecheck{h}^{\mathrm{par}}$ is equivariant with respect to the scaling $\BG_m$-action on the Higgs fields, it suffices to bound from above the dimension of the fiber over $(x, 0) \in C\times \widehat{A}$ for each point $x \in C$ by $\mathbf{d}$, since upper semicontinuity will force all other fibers to have the same dimension upper-bound.\footnote{Furthermore, we actually show that each fiber has the same dimension. But the dimension bound is enough for our purpose.}  We fix the point $p$ from now on.

Using the assumption $\deg(\CL)>2g$, we may express $\CL$ as $\CO_C(D)$ with $D = x_0 + x_1 + \dots + x_t$ an effective reduced divisor containing $x = x_0$. In particular, a Higgs bundle $(\CE, \theta: \CE \to \CE\otimes \Omega_\CL)$ can be viewed as a meromorphic (un-twisted) Higgs bundles with at most simple poles along $D$:
\[
(\CE, \theta), \quad \theta: \CE \to \CE \otimes \Omega_C^1(D).
\]
We will control the fiber dimension using an analogous bound for the nilpotent cone for \emph{strongly parabolic} Higgs bundles for the pair $(C,D)$.

Recall that, given $D$ as above, a strongly parabolic Higgs bundle consists of an $\mathrm{SL}_n$ Higgs bundle of degree $d$:
\[
(\CE, \theta), \quad \theta: \CE\to \CE \otimes \Omega^1_C(D),~~~\mathrm{rk}(\CE) =n,~~\mathrm{det}(\CE)\simeq \CN, ~~\mathrm{trace}(\theta) = 0,
\]
along with a flag at each point $x_i \in D$:
\[
\CE|_{x_i}=F_{x_i,0}  \supset F_{x_i,1} \supset F_{x_i,2}, \supset \cdots \supset F_{x_i,n} = 0, \quad \mathrm{dim} F_{x_i,j}/F_{x-i, j+1} = 1,
\]
such that the Higgs field satisfies $\theta(F_{x_i, j}) \subset F_{x_i, j+1}$.\footnote{One may compare this with the (non-strong) parabolic condition $\theta(F_{x,j}) \subset F_{x,j}$ we used earlier in this paper for the global Springer theory.}


Let $\widecheck{M}^{\mathrm{spar}}(D)$ denote the moduli space of strongly parabolic Higgs bundles associated with the pair $(C,D)$, such that the underlying twisted Higgs bundle is stable. Let
$$A(D) = \bigoplus_{i=2}^{n} H^0(C, \Omega_\CL^{\otimes i}(-D)) \subset \widehat{A}$$
denote the linear subspace consisting of sections which vanish along the points of $D$. There is a Hitchin map for the strongly parabolic Higgs moduli space
$$\widecheck{h}^{\mathrm{spar}}:\widecheck{M}^{\mathrm{spar}}(D) \rightarrow A(D)$$
which is proper, surjective, and Lagrangian with respect to a natural holomorphic symplectic form \cite{Faltings}. In particular, the dimension of the zero fiber is given by
\begin{align*}
\dim \widecheck{M}^{\mathrm{spar}}(D)_0 = \dim A(D) &= \sum_{j=2}^{n} \left((j-1)(\deg(\Omega_L)) +g-1\right)\\
&= \frac{n(n-1)}{2}\deg (\CL) + (n^2-1)(g-1).
\end{align*}
See also \cite[Theorem 6.9]{SWW} for a direct proof of the above dimension formula.\footnote{The formula above has $g$ less than the dimension formula obtained in \cite[Theorem 6.9]{SWW}, since we consider the $\mathrm{SL}_n$ case where we fixed the determinant on $C$.}

We now consider the following diagram relating the two types of parabolic Higgs moduli spaces $\widecheck{M}^{\mathrm{spar}}(D)$ and 
$\widecheck{M}^{\mathrm{par}}$:

\begin{equation}\label{diagram}
\begin{tikzcd}
\widecheck{M}^{\mathrm{spar}}(D)_0 \arrow{d}{q_0}\arrow[r,hook]
&\widecheck{M}^{\mathrm{spar}}(D)\arrow{d}{q}
& 
\\
\widecheck{M}^{\mathrm{par}}_{(x,0)}\arrow{d}\arrow[r, hook]&
{\widecheck{M}^{\mathrm{par}}}|_{x \times A(D)} \arrow{d}\arrow[r,hook]&
\widecheck{M}^{\mathrm{par}}\arrow[d,"\widecheck{h}^{\mathrm{par}}"]\\
\{(x,0)\} \arrow[r, hook] & x \times A(D)  \arrow[r, hook] & C\times \widehat{A}
\end{tikzcd}
\end{equation}
where all squares are Cartesian.  

We first claim that the natural map $q$ (sending a strong parabolic Higgs bundle to a parabolic Higgs bundle) is surjective. Indeed, given a point $z \in \widecheck{M}^{\mathrm{par}}$ lying over $C\times {A}(D)$,
choosing a point in its preimage $q^{-1}(z)$ only consists of fixing a flag at each point of $x_i$ with $i>0$, preserved by the Higgs field --- since the characteristic polynomial already has zero roots at all points of $D$, the Higgs field is automatically strongly parabolic.

By base change, this implies that $q_0$ is surjective as well, so we get the desired dimension upper bound
\[
\dim {\widecheck{M}^{\mathrm{par}}}_{(x,0)} \leq \dim \widecheck{M}^{\mathrm{spar}}(D)_0 = \frac{n(n-1)}{2}\deg (\CL) + (n^2-1)(g-1),
\]
which completes the proof. \qed

\begin{rmk}
Combining Sections 3 and 4 proves that Theorem \ref{conj2.7} holds for a line bundle $\CL$ of sufficiently large degree. Then the argument of Section \ref{Sec2.3} implies that it actually holds for any effective $\CL$.
\end{rmk}

\end{document}